\documentclass[12pt]{article}

\usepackage{xypic}
\usepackage{comment}
\usepackage{amsmath,amstext,amsbsy,amssymb,amsthm}
\usepackage{mine}
\begin{document}
\selectfont
\nocite{Kohl1996}
\define\div{\big|}
\define\ndiv{\not\div}
\define\QI{ \Bbb Q(\zeta_{\infty}) }
\define\QWI{ \Bbb Q(w_{\infty})}
\define\QWII{ \Bbb Q(w_{\infty},\zeta_{\infty}) }
\define\DI{ \Delta_{\infty} }
\define\NI{N_{\infty}}
\define\HI{ H_{\infty}}
\define\TI{\QI\otimes_{\Bbb Q}}
\define\DSG{\underset{i\geq 1}{\oplus} \mathbb{Q}(\zeta_{p^i})N_{p^i}}
\define\DSH{\underset{i\geq 1}{\oplus} H_{p^i}}
\define\TPL{(\dots,0,b_{n},b_{n-1},\dots,b_{2},b_{1})}
\define\ZPL{(\dots,0,b_{k},0,\dots,0,0)}
\define\DTPL{(\dots,0,{}^{\dl_{p^n}} b_{n} ,{}^{\dl_{p^{n-1}}}b_{n-1},\dots,{}^{\dl_{p^2}}b_{2},{}^{\dl_{p}}b_{1})}
\define\DDTPL{(\dots,0,{}^{\dl_{p^n}^{e_{n}}} b_{n} ,{}^{\dl_{p^{n-1}}^{e_{n-1}}}b_{n-1},\dots,{}^{\dl_{p^2}^{e_{2}}}b_{2},{}^{\dl_{p}^{e_1}}b_{1})}
\define\sg{\sigma}
\define\dl{\delta}
\def\keywords#1{\def\@keywords{#1}}
\parskip=0.125in
\title{A Class of Profinite Hopf-Galois Extensions over Q}
\author{Timothy Kohl\\
Department of Mathematics and Statistics\\
Boston University\\
Boston, MA 02215\\
tkohl@math.bu.edu}
\maketitle
\begin{abstract}
For $p$ a prime and $a\in\mathbb{Q}$, where $a$ is not a $p^n$-th power of any rational number, the extension $\mathbb{Q}(w_n)/\mathbb{Q}$ where $w_n=\root p^n \of a$ is separable but non-normal. The Hopf-Galois theory for separable extensions was determined by Greither and Pareigis, and the specific classification for radical extensions such as these by the author. In this work we extend this theory to a certain class of profinite extensions, namely those formed from the union of these $\mathbb{Q}(w_n)$. We construct a 'profinite' Hopf algebra which acts, and show that it satisfies a generalization of a result due to Haggenm{\"u}ller and Pareigis on the structure of Hopf algebra forms of group algebras. 
\end{abstract}
\par
\renewcommand{\thefootnote}{}\centerline{Introduction}
A separable field extension $K/k$ is Hopf-Galois if there is $k$-Hopf algebra $H$ as well as a $k$-algebra map $\mu:H\rightarrow End_k(K)$ such that for $h\in H$ and $a,b\in K$, one has
\begin{equation}
\label{measuring}
\mu(h)(ab)=\sum_{(h)}\mu(h_{(1)})(a)\mu(h_{(2)})(b)
\end{equation}
where for $\Delta:H\leftarrow H\otimes H$ the co-multiplication of $H$,
$$
\Delta(h)=\sum_{(h)}h_{(1)}\otimes h_{(2)}
$$
and where $\mu$ induces a $k$-algebra isomorphism
\begin{equation}
\label{linindchar}
1\#\mu:K\#H\rightarrow End_k(K)
\end{equation}
and
\begin{equation}
\label{fixedfield}
K^H=\{x\in K | \mu (h)(x)=\epsilon (h)x \ \forall h\in H\}=k
\end{equation}
where $\epsilon:H\rightarrow k$ is the co-unit map of $H$. All these properties generalize what happens when $K/k$ is Galois with $G=Gal(K/k)$ for then one may set $H=k[G]$ with $\Delta(g)=g\otimes g$ and $\epsilon(g)=1$ and where the map $\mu:H\rightarrow End_k(K)$ is obvious. In this situation \eqref{measuring} is simply the multiplicative action of the group elements extended linearly to sums in $H$, \eqref{fixedfield} is due to $K^G=k$ and \eqref{linindchar} is linear independence of characters. As such, a Hopf-Galois structure for some other Hopf algebra (where indeed $K/k$ may not even be Galois in the first place) is a way to generalize these three fundamental aspects of ordinary (classical) Galois theory of fields. Indeed the case where $K/k$ is separable, but {\it not} Galois is the starting point of the work of Greither and Pareigis \cite{GP}. The primordial example they begin with (and one which is, broadly generalized, part of the family of extensions we consider in this work) is $\mathbb{Q}(\root 3 \of 2)/\mathbb{Q}$ which is non-normal due to the lack of roots of unity in the ground field $\mathbb{Q}$ of course. Nonetheless, a certain rank 3 Hopf algebra can be shown to act on this extension to make it Hopf-Galois. Rather than give the example as it is presented in \cite{GP} we shall look at all such radical extensions in general.\par
\section{Greither-Pareigis Theory}
The setup in \cite{GP} is as follows. For $K/k$ a finite separable extension of fields, with Galois closure $\tilde{K}/K$, let $\Gamma=Gal(\tilde{K}/k)$ and $\Delta=Gal(\tilde{K}/K)$. The natural action of $\Gamma$ on the left cosets $S=\Gamma/\Delta$ yields a map $\lambda:\Gamma\rightarrow Perm(S)=B$. Note that for $\Delta$ trivial $\lambda$ is the left regular representation of $\Gamma$ in its group of permutations. Also one needs the following definition.
\begin{definition}
A regular subgroup $N\leq Perm(S)$ is one that acts transitively and fixed point freely on the elements of $S$. That is the orbit of any element of $S$ under the action of $N$ is all of $S$ and if $n(s)=s$ for any $s\in S$ then $n$ is the identity element. 
\end{definition}
A consequence of regularity is that for any such $N$, $|N|=|S|$. With these definitions in mind, one has
\begin{theorem}{\cite[Theorem 2.1]{GP}}
\label{GPmain}
Given $K/k$ and $B=Perm(\Gamma/\Delta)$ for $\Gamma$, $\Delta$ as above, the following are equivalent:\par
\noindent (a) There is a $k$-Hopf algebra $H$ making $K/k$ $H$-Galois.\par
\noindent (b) There is a regular subgroup $N\leq B$ such that $\lambda(\Gamma)\leq Norm_B(N)$\par
\noindent Moreover by Galois descent $H$ can shown to be $(\tilde{K}[N])^{\Gamma}$, the fixed ring under the diagonal (simultaneous) action of $\Gamma$ on $\tilde{K}$ (via the Galois action) and on $N$ by conjugation in $B$ by the elements of $\lambda(\Gamma)$.
\end{theorem}
The Hopf algebra given as the fixed ring has the further property that $\tilde{K}\otimes H\cong \tilde{K}\otimes k[N]$, that is, in the language of descent, $H$ is a $\tilde{K}$-form of the group ring $k[N]$, i.e. the two $k$-Hopf algebras become isomorphic as $\tilde{K}$ Hopf algebras. The enumeration of Hopf-Galois structures on a given $K/k$ amounts to the enumeration of the regular subgroups of $B$ normalized by $\lambda(\Gamma)$. Much recent work has focused on the case where $K/k$ is already Galois with group $\Gamma$ and therefore one is searching for regular subgroups of $Perm(\Gamma)$ normalized by the left-regular representation of $\Gamma$. However, we shall be considering cases closer in spirit to the original motivation for this subject, namely the construction of Hopf-Galois structures on non-normal separable extensions. Of particular relevance to the current discussion is the following class of extensions as defined in \cite[Proposition 4.1]{GP}
\begin{definition}
\label{almostclassical}
If $K/k$ is a separable extension with $\Gamma$ and $\Delta$ as above then if there exists $N\triangleleft\Gamma$ such that $N\cap\Delta=\{e\}$ and $\Gamma=N\Delta$ (i.e. $N$ is a normal complement to $\Delta$ in $\Gamma$) then $K/k$ is Hopf-Galois for $H=(E[N^{opp}])^{\Delta}$ where $E=\tilde{K}^{N}$ and $N^{opp}=Cent_B(N)$. Such a Hopf-Galois structure is termed {\it almost classical}.
\end{definition}
This bears some exploration in terms of the construction of the Hopf algebra which acts. If $N$ is a normal subgroup of $\Gamma$ with fixed field $E$ then have
$$
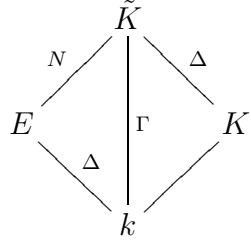
\diagram
& \tilde{K} \dlline_{N} \ddline^{\Gamma} \drline^{\Delta} \\
E \drline^{\Delta} & &K \dlline \\
&k
\enddiagram$$
where $Gal(E/k)\cong\Delta$ by natural irrationality. Observe that $N$ being a normal subgroup of $\Gamma$ means that, if we identify $N$ with $\lambda(N)$ embedded in $B$ that $\lambda(N)$ is normalized by $\lambda(\Gamma)$. Moreover, since $N$ is a normal complement to $\Delta$ in $\Gamma$ then (viewed inside $B$), $N$ is regular. As such, $N$ itself gives rise to a Hopf-Galois structure with Hopf algebra $(\tilde{K}[N])^{\Gamma}$. However, for $N$ a regular subgroup it's readily shown that $N^{opp}=Cent_B(N)$ is also regular and moreover that $Norm_B(N)=Norm_B(N^{opp})$ which is a direct analogue of the relationship between the left and right regular representations of a group $G$ in its group of permutation (whose common normalizer is the holomorph $Hol(G)$). The Hopf algebra that arises using $N^{opp}$ is of course $(\tilde{K}[N^{opp}])^{\Gamma}$, but since $\Gamma=N\Delta$ and since $N$ centralizes $N^{opp}$ and has fixed field $E$ then
$$
(\tilde{K}[N^{opp}])^{N\Delta}=((\tilde{K}[N^{opp}])^N)^{\Delta}=(E[N^{opp}])^{\Delta}
$$
which is basically the observation made in \cite[Corollary 3.2]{GP}. For the case where $N$ is abelian then $N^{opp}=N$ and $H=(E[N])^{\Delta}$ whose action even more so merits the adjective {\it almost classical} since the action is based on the existence and action of a subgroup of the Galois closure of $K/k$. For the cases we shall be studying this will be the situation since $N$ will be cyclic of prime power order. Indeed it is a natural extension of the author's work in \cite{Kohl1998} enumerating the Hopf-Galois structures on radical extensions $k(w)/k$ with $w\not\in k$ where $char(k)=0$ and $w^{p^n}=a\in k$ ($p$ an odd prime) where $k$ contains at most a ${p^r}$-th root of unity but not a $p^{r+1}$ root of unity. 
\begin{theorem}{\cite[Theorem 3.3 and Theorem 4.5]{Kohl1998}}
\label{mine}
 The radical extension $k(w)/k$ given above has exactly $p^r$ Hopf-Galois structures for $0\leq r< n$ and $p^{n-1}$ for $r=n$, of which $p^{min(r,n-r)}$ are almost classical and for all, the associated group $N$ is cyclic of order $p^n$. 
\end{theorem}
\noindent Note that for $r=0$ such a radical extension has exactly one Hopf-Galois structure, and it is almost classical.
\section{Radical Extensions of Q}
Here we shall consider the radical extensions $\mathbb{Q}(a^{1/ p^n}) / \mathbb{Q}$ for $p$ an odd prime, where $a$ is not a $p^{n}$-th power of any rational number. Since $\mathbb{Q}$ contains no $p^{n}$-th roots of unity, these extensions are acted on by a unique Hopf algebra. We shall discuss their structure and how they relate to Greither Pareigis theory and the study of Hopf algebra forms in \cite{HP}. Later, we shall construct a profinite Hopf algebra form that generalizes Theorem 5 of \cite{HP} and show that this form acts on the direct limit (union) of the extensions $\mathbb{Q}(a^{1/p^n})$. We will also view this action in terms of regularity to try and extend the results of Greither and Pareigis to such profinite separable extensions.\par
\noindent We shall use the following notation:
\begin{align*}
w_{n} &= a^{1/p^n}, a \in \mathbb{Q} \  ( a \text{ not already a } p^{n-th} \text{ power}) \\
\zeta_{n} &= \text{ a  primitive }  p^{n-th} \text{  root  of  unity}
\end{align*}
\begin{align*}
N_{n} &= Gal(\mathbb{Q}(\zeta_{n},w_{n})/\mathbb{Q}(\zeta_{n})) \\
\Delta_{n} &= Gal(\mathbb{Q}(\zeta_{n},w_{n})/\mathbb{Q}(w_{n})) \\
\Gamma_{n} &= Gal(\mathbb{Q}(w_{n},\zeta_{n}) / \mathbb{Q})
\end{align*}
which is diagrammed below.
$$
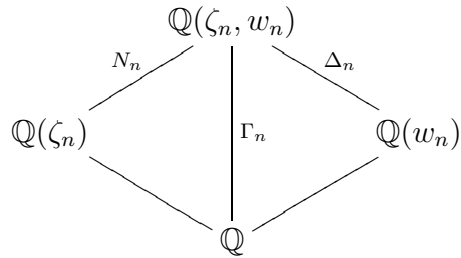
\diagram
& \mathbb{Q}(\zeta_{n},w_{n}) \dlline_{N_{n}} \ddline^{\Gamma_{n}} \drline^{\Delta_{n}} \\
\mathbb{Q}(\zeta_{n}) \drline & &\mathbb{Q}(w_{n}) \dlline \\
&\mathbb{Q}
\enddiagram$$
Now $\mathbb{Q}(w_{n})/\mathbb{Q}$ is a separable, non-normal extension which is $H$-Galois, for $H$ exhibited below. Observe that $\mathbb{Q}(\zeta_{n})/\mathbb{Q}$ is such that $\mathbb{Q}(\zeta_{n}) \mathbb{Q}(w_{n})  = \  \mathbb{Q}(\zeta_{n},w_{n})$ the normal closure of  $\mathbb{Q}(w_{n})/\mathbb{Q}$.  Hence  $\mathbb{Q}(w_{n})/\mathbb{Q}$
is almost classically Galois.
To compute the relevant Hopf algebra form we proceed as follows.
There is a natural action:
$$\mathbb{Q} ( \zeta_{n})[N_{n}]\otimes\mathbb{Q} (\zeta_{n},w_{n})\longrightarrow\mathbb{Q} (\zeta_{n},w_{n})$$ 
Here $N_{n} =  \langle \sg_{n} \rangle $ is cyclic of order $p^n$ where 
$\sg_{n}(w_{n})=\zeta_{n}w_{n}$ and $\Delta_{n} \cong Gal(\mathbb{Q}(\zeta_{n})/\mathbb{Q})$. This, in turn, is isomorphic to $(\Bbb Z / p^n \Bbb Z)^{\times}$
via the homomorphism $t:\Delta_{n}\longrightarrow (\Bbb Z/p^n\Bbb Z)^{\times}$ 
where, since $\Delta_{n} =  \langle \dl_{n} \rangle $ is cyclic of order $\phi(p^n)$,
we define $t(\dl_{n}) = \pi\in(\Bbb Z/p^n\Bbb Z)^{\times}$ a primitive root
mod $p^n$. Furthermore $\Delta_{n}\cong Aut(N_{n})$ via the map $\tau_{n}$ where we define $\tau_{n}({\dl_{n}})(\sg_{n}^{i}) = \sg_{n}^{it(\dl_{n})} =
\sg_{n}^{i\pi}$. Moreover, since $\mathbb{Q}(w_{n})$ and $\mathbb{Q}(\zeta_{n})$ are
linearly disjoint over $\mathbb{Q}$ we have 
$\Gamma_{n} = N_{n} \Delta_{n}\cong N_{n}\rtimes\Delta_{n}$ and by natural irrationality $\Delta_{n}$ may be viewed as $Gal(\mathbb{Q}(\zeta_{n})/\mathbb{Q})$. We can define then an action (the diagonal action) of $\Delta_{n}$ on the group ring
$\mathbb{Q}(\zeta_{n})[N_{n}]$ as follows:
$${}^{\dl_{n}}(r\zeta_{n}^{a}\sg_{n}^{b}) = r\zeta_{n}^{a\pi}\sg_{n}^{b\pi} \ where \ r\in\mathbb{Q}$$\par
\noindent Of course, $\Delta_{n}$ acts on $\mathbb{Q}(\zeta_{n},w_{n})$, so we may
pass from: $$\mathbb{Q}(\zeta_{n})[N_{n}]\otimes\mathbb{Q}(\zeta_{n},w_{n})\longrightarrow\mathbb{Q}(\zeta_{n},w_{n})$$
by descent to the corresponding action:$$(\mathbb{Q} (\zeta_{n})[N_{n}])^{\Delta_{n}}\otimes(\mathbb{Q} (\zeta_{n},w_{n}))^{\Delta_{n}}\longrightarrow(\mathbb{Q}(\zeta_{n},w_{n}))^{\Delta_{n}}$$      
By Theorem 5 of \cite{HP} we have that $(\Bbb  Q(\zeta_{n})[N_{n}])^{\Delta_{n}}$ 
is a $\mathbb{Q}(\zeta_{n})$-Hopf algebra form of $\mathbb{Q} N_{n}$, since
$\Delta_{n}=Gal(\mathbb{Q}(\zeta_{n})/ \mathbb{Q}))$ and $Aut(N_{n})\cong\Delta_{n}$.
We shall denote this Hopf algebra by $H_{n}$ and since
$\mathbb{Q}(\zeta_{n})\otimes_{\mathbb{Q}} H_{n}\cong \mathbb{Q}(\zeta_{n})[N_{n}]$,
then $rk_{\mathbb{Q}}(H_{n})$ = $p^n$.
That we have an action
$$H_{n}\otimes\mathbb{Q}(w_{n})\longrightarrow\mathbb{Q}(w_{n})$$ follows from
\cite[Proposition 4.1]{GP} which says that the almost classical extension 
$\mathbb{Q}(w_{n})/ \mathbb{Q}$ is Hopf-Galois for an $E$-Hopf algebra form of a
$\mathbb{Q} N$ where $N=Gal(E/\mathbb{Q})$ is a normal complement to $\Delta_{n}$
inside $\Gamma_{n}$. However, again by \cite[Theorem 3.7 and Theorem 4.5]{Kohl1998}, the only such subgroup is $N_{n}$ 
and so this is the only almost-classical structure for this extension.
\subsection{$H_{n}$ and how it acts}
The Hopf algebra $H_{n}$ is a fixed ring under the action of $\Delta_{n}$ as given above. We shall show that this is a relatively familiar object by constructing a basis for it.\par
\begin{proposition} 
\label{Hbasis}
$H_{n}$ is isomorphic to $(\mathbb{Q}N_{n})^*$ the linear dual of the group ring.
\end{proposition}
\begin{proof} 
\noindent We start by defining the elements:
$$e_{n,i}=\frac{1}{p^n}\sum_{j=0}^{p^n-1}\zeta_n^{-ij}\sg_n^{j}$$
for $i$ from 0 to $p^n-1$. For $\Delta_{n}=\langle\delta_n\rangle$ we have
\begin{align*}
{}^{\delta_n}e_{n,i} &= \frac{1}{p^n}\sum_{j=0}^{p^n-1}\zeta_n^{-\pi ij}\sg_n^{\pi j} \\ 
	       &= \frac{1}{p^n}\sum_{j=0}^{p^n-1}\zeta_n^{-i(\pi j)}\sg_n^{(\pi j)} \\
	       &= e_{n,i}
\end{align*}
since $(\pi,p)=1$. Therefore each $e_{n,i}$ is contained in $H_{n}$ and to show that these comprise a basis for $H_{n}$
we shall, along the way, identify $H_{n}$ as a familiar object. Specifically, given $N_{n}=\langle\sg\rangle$ we have the character group $\widehat{N}_{n}=\langle\chi_n\rangle$ where $\chi_n^{j}(\sg_n^{k})=\zeta_n^{jk}$. In \cite[Theorem 7.10]{SC} it is shown that for a connected commutative ring $R$, and finite abelian group $G$ where $|G|$ is invertible in $R$ that if $R$ contains a primitive $exp(G)^{th}$ root of unity that $R[G]\cong R[\hat{G}]\cong (R[G])^{*}=Hom_R(R[G],R)$. Our basis for $H_n$ uses basically this result. Since $N_{n}$ is cyclic, the map $\sg_n\mapsto\chi_n$ is a group isomorphism which can be extended by linearity to an isomorphism (of rings)
$\mathbb{Q}(\zeta)[N_{n}]\rightarrow \mathbb{Q}(\zeta) \widehat{N}_{n}$. If we let
$\Delta_{n}$ act on $\mathbb{Q}(\zeta)$ in the usual way and act on
$\widehat{N}_{n}$ by ${}^{\delta_n}\chi_n=\chi_n^{\pi}$ then
the above isomorphism is $\Delta_{n}$-equivariant. Under this map, 
$e_{n,i}\mapsto\widehat{e}_{n,i}$ where
$\widehat{e}_{n,i}=\frac{1}{p^n}\sum_{j=0}^{p^n-1}\zeta_n^{-ij}\chi_n^{j}$ and we have then that
\begin{align*}
\widehat{e}_{n,i}(\sg_n^{k}) &= \frac{1}{p^n}\sum_{j=0}^{p^n-1}\zeta_n^{-ij}\chi_n^{j}(\sg_n^{k}) \\
		       &= \frac{1}{p^n}\sum_{j=0}^{p^n-1}\zeta_n^{-ij}\zeta_n^{kj}\\
		       &= \frac{1}{p^n}\sum_{j=0}^{p^n-1}\zeta_n^{j(k-i)} \\
		       &= \delta_{ik}
\end{align*}
That is, we may identify the $\widehat{e}_{n,i}$'s with the $e_{n,i}$ constructed earlier and 
in doing so identify the $\mathbb{Q}$ span of these with $(\mathbb{Q} N_{n})^{*}$,
the $\mathbb{Q}$ dual of the group ring $\mathbb{Q} N_{n}$. Since the isomorphism
$\mathbb{Q}(\zeta)[N_{n}]\rightarrow \mathbb{Q}(\zeta)\widehat{N}_{n}$ is $\Delta_{n}$-equivariant
we conclude that $H_{n}\cong (\mathbb{Q} N_{n})^{*}$. \par
\end{proof}
\noindent This isomorphism is not unexpected in light of \cite[p.247 Remark 2]{GP} where the authors observe (due to \cite[p.39]{CS}) that if $X^n-a^n$ is irreducible over $k$ then $k(a)/k$ is Hopf-Galois where the Hopf algebra acting is the dual of the group ring. From \ref{mine} we know that this is, of course, the {\it only} Hopf-Galois structure. The isomorphism of $H_{n}$ with the dual of the group ring is not just merely a way to identify it with something familiar. It actually yields an interesting parallel when we look at how it acts on $\mathbb{Q}(w_{n})$.
\begin{proposition}
\label{Hact}
The action of $H_{n}$ on $\mathbb{Q}(w_{n})$ is as follows. If $i=0,\dots,p^n-1$ and $k=0,\dots p^{n-1}$ then $e_{n,i}(w_{n}^k)=\delta_{ik}w_{n}^k$.
\end{proposition}
\begin{proof}
A basis for $\mathbb{Q}(w_{n})$ over $\mathbb{Q}$ consists of powers $w_{n}^k$ for $k$ from $0$ to $p^n-1$. Using the $e_{n,i}$ given earlier, where $i=0,\dots,p^{n}-1$, direct calculation yields
\begin{align*}
e_{n,i}(w_{n}^k) &= \frac{1}{p^n}\sum_{j=0}^{p^n-1}\zeta_{n}^{-ij}\sg_n^{j}(w_{n}^k) \\ 
         &= \frac{1}{p^n}\sum_{j=0}^{p^n-1}\zeta_{n}^{-ij}(\zeta_{n}^{kj}w_{n}^{k}) \\
         &= w_{n}^k\left[\frac{1}{p^n}\sum_{j=0}^{p^n-1}\zeta_{n}^{-ij+kj}\right] \\
         &= w_{n}^k\left[\frac{1}{p^n}\sum_{j=0}^{p^n-1}\zeta_{n}^{j(k-i)}\right] \\
         &= \delta_{ik}w_{n}^k\\
\end{align*}
\end{proof}
\noindent As such, the $e_{n,i}$ are almost a 'dual basis' to $\{1,w_{n},\dots,w_{n}^{p^n-1}\}$.\par
\noindent Now we know that $H_n$ will be a form of $\mathbb{Q}[N_n]$ in that $\mathbb{Q}(\zeta_n)\otimes H_n\cong \mathbb{Q}(\zeta_n)[N_n]$ but it will also be important in the sequel to have some insight into the structure of $\mathbb{Q}(\zeta_m)\otimes H_n$ for different $m$. We have the following:
\begin{lemma}
\label{QmtensorHn}
Given $H_n$ as defined above, if $m\geq n$ then $\mathbb{Q}(\zeta_m)\otimes H_n\cong \mathbb{Q}(\zeta_m)[N_n]$ and if $m<n$ then $\mathbb{Q}(\zeta_m)\otimes H_n$ contains $\sigma_n^{p^{n-m}}$.
\end{lemma}
\begin{proof}
For $m\geq n$ we have that $\mathbb{Q}(\zeta_n) \subseteq \mathbb{Q}(\zeta_m)$ so that $\mathbb{Q}(\zeta_m) \otimes H_n=\mathbb{Q}(\zeta_m)\otimes (\mathbb{Q}(\zeta_n)\otimes H_n)=\mathbb{Q}(\zeta_m)\otimes (\mathbb{Q}(\zeta_n)[N_n])=\mathbb{Q}(\zeta_m)[N_n]$.\par
If $m<n$ then $\mathbb{Q}(\zeta_m)\otimes H_n$ will not be the full group-ring since $\mathbb{Q}(\zeta_m)$ doesn't contain $p^{n-th}$ roots of unity, so in particular it will not contain $\sigma_n$. However, using the basis $\{e_{n,i}\}$ for $H_n$ together with the fact that $\zeta_m=(\zeta_n)^{p^{n-m}}$ we can show that this partial base extension of $H_n$ contains $\sigma_n^{p^{n-m}}$. Consider the following $\mathbb{Q}(\zeta_m)$-linear combination of the $e_{n,i}$
\begin{align*}
\sum_{i=0}^{p^n-1}\zeta_n^{a_ip^{n-m}}e_{n,i}&=\sum_{i=0}^{p^n-1}\zeta_n^{a_ip^{n-m}}\left[\frac{1}{p^n}\sum_{j=0}^{p^n-1}\zeta_n^{-ij}\sigma_n^j\right]\\
&=\sum_{j=0}^{p^n-1}\left[\frac{1}{p^n}\sum_{i=0}^{p^n-1}\zeta_n^{a_ip^{n-m}}\zeta_n^{-ij}\right]\sigma_n^j\\
&=\sum_{j=0}^{p^n-1}\left[\frac{1}{p^n}\sum_{i=0}^{p^n-1}\zeta_n^{a_ip^{n-m}-ij}\right]\sigma_n^j\\
\end{align*}
where now the coefficient of $\sigma_n^{p^{n-m}}$ is 
$$
\frac{1}{p^n}\sum_{i=0}^{p^n-1}\zeta_n^{(a_i-i)p^{n-m}}
$$
which, if we choose $a_i=i$ for each $i$ yields $1$. And for $j\neq p^{n-m}$ one has
$$
\frac{1}{p^n}\sum_{i=0}^{p^n-1}\zeta_n^{i(p^{n-m}-j)}
$$
where we may view the $i$ as coming from $\mathbb{Z}_{p^n}$. As such if $p^{n-m}-j=p^ku$ where $gcd(u,p)=1$ then multiplication by $p^ku$ represents an onto homomorphism from $\mathbb{Z}_{p^n}\rightarrow\mathbb{Z}_{p^{n-k}}$ and since 
$$\sum_{t\in\mathbb{Z}_{p^{n-k}}}\zeta_n^{t}=0$$
then 
$$
\frac{1}{p^n}\sum_{i=0}^{p^n-1}\zeta_n^{i(p^{n-m}-j)}=0
$$
for $j\neq p^{n-m}$. Thus, this $\mathbb{Q}(\zeta_m)$-linear combination of the $e_{n,i}$ is exactly $\sigma^{p^{n-m}}$. That $\mathbb{Q}(\zeta_m)\otimes H_n$ contains the unique order $p^m$ subgroup of $N_n$ is not a coincidence since one could, from the elements of $\langle\sigma_{n}^{p^{n-m}}\rangle$ and $\mathbb{Q}(\zeta_m)\subseteq \mathbb{Q}(\zeta_n)$ form a collection $\{e_{m,i}'\}$ whose $\mathbb{Q}$-span would be an $H'_m\subseteq H_n$ isomorphic to $H_m$ where therefore $\mathbb{Q}(\zeta_m)\otimes H'_m\cong\mathbb{Q}(\zeta_m)[N_m]$. The point is, that this is the smallest subgroup of $N_n$ which lies in $\mathbb{Q}(\zeta_m)\otimes H_n$.
\end{proof}

\noindent It's also interesting to compare the extensions $\mathbb{Q}(\zeta_{n},w_{n})/\mathbb{Q}(\zeta_{n})$ and $\mathbb{Q}(w_{n})/\mathbb{Q}$. The extension $\mathbb{Q}(\zeta_{n},w_{n})/\mathbb{Q}(\zeta_{n})$ is Galois with respect to the group $N_{n}$ and therefore canonically Hopf-Galois with respect to the Hopf algebra $\mathbb{Q}(\zeta_{n})[N_{n}]$. And we've now demonstrated that the extension $\mathbb{Q}(w_{n})/\mathbb{Q}$ is Hopf-Galois with respect to $H_{n}$ which is isomorphic to $(\mathbb{Q}N_{n})^{*}$. The analogy being made here is to the natural irrationality $Gal(\mathbb{Q}(\zeta_{n},w_{n})/\mathbb{Q}(w_{n}))\cong Gal(\mathbb{Q}(\zeta_{n})/\mathbb{Q})$.\par
It is also worth considering the induced isomorphism (in this case) of 
$$\mathbb{Q}(w_{n})\# H_{n}\cong End_\mathbb{Q}(\mathbb{Q}(w_{n}))$$
which is a consequence of $\mathbb{Q}(w_{n})/\mathbb{Q}$ being Hopf-Galois with respect to $H_{n}$. The underlying algebra of $\mathbb{Q}(w_{n})\# H_{n}$ is $\mathbb{Q}(w_{n})\otimes H_{n}$ but where the multiplication is 'twisted' by the action of $H_{n}$ on $\mathbb{Q}(w_{n})$. Specifically
\begin{align*}
(a \# h)(b \# h')&=\sum_{(h)}a h_{(1)}(b)\# h_{(2)}h'\\
\text{where }\Delta(h)&=\sum_{(h)}h_{(1)}\otimes h_{(2)}\\
\end{align*}
and for $H_{n}$ one has 
$$
\Delta(e_{n,i})=\sum_{\{s,t\ |\ s+t=i\}} e_{n,s}\otimes e_{n,t}
$$
since $H_{n}$ is dual to the group ring and the $e_{n,i}$ are the basis of this dual. Bear in mind also that 
\begin{align*}
dim_{\mathbb{Q}}(End_{\mathbb{Q}}(\mathbb{Q}(w_{n})))&=dim_{\mathbb{Q}}(\mathbb{Q}(w_{n})\otimes H_{n})) \\
                                                 &=[\mathbb{Q}(w_{n}):\mathbb{Q}]\cdot dim_{\mathbb{Q}}(H_{n})\\
                                                 &=p^{2n}
\end{align*}
where $H_{n}$ is embedded as the span of the $\{e_{n,i}\}$ given above in \ref{Hact} and $\mathbb{Q}(w_{n})$ is embedded as those linear transformations induced by left multiplication by the basis elements $\{1,w_n,\dots,w_{n}^{p^n-1}\}$. As such $\{w_{n}^j\# e_{n,i}\}$ is a basis for $\mathbb{Q}(w_{n})\# H_{n}$, where we may also view $End_{\mathbb{Q}}(\mathbb{Q}(w_{n}))$ as being spanned by these elements. Specifically we have 
\begin{equation}
\label{endact}
(w_{n}^j\# e_{n,i})(w_{n}^k)=\begin{cases} 0 \text{\ \ \ \ \ \ if $i\neq k$}\\ w_{n}^{j+k}\text{ if $i=k$}\end{cases}
\end{equation}
which yields the multiplication explicitly, in accordance with the formula above:
\begin{align*}
(w_{n}^j\# e_{n,i})(w_{n}^k\# e_{n,l})&=\sum_{\{s,t\ s+t=l\}}w_{n}^j e_{n,s}(w_{n}^k)\#e_{n,t}e_{n,l}\\
                              &=\sum_{\{s,t\ s+t=l\}}w_{n}^j (\delta_{s,k}w_{n}^k)\#\delta_{t,l}e_{n,t}\\
                              &=\begin{cases} w_{n}^{j+k}\# e_{n,l}\text{ if $k+l=i$}\\ 0\text{\ \ \ \ \ \ \ \ \ \  otherwise}\end{cases}
\end{align*}
As an interesting computational sideline, there is a nice way to associate the actions of $e_{n,i}$ and $w_{n}^j$ within $End_{\mathbb{Q}}(\mathbb{Q}(w_n))$ as matrices and the $w_{n}^j\# e_{n,i}$ as products of these matrices. We demonstrate this explicitly in the case $p=3$ and $n=1$.\par
Viewing $\{1,w,w^2\}$ as the basis for $\mathbb{Q}(w)$, each $e_{i}$ can be represented as a $3\times 3$ matrix which is zero except for the ${i+1}^{st}$ column which consists of the ${i+1}^{st}$ elementary basis vector for $V=\mathbb{Q}^3$. i.e. We're making the identification $End_{\mathbb{Q}}(\mathbb{Q}(w))\cong End_{\mathbb{Q}}(V)\cong GL_3(\mathbb{Q})$. Similarly, we view $w^i$ as left multiplication $l_{w^i}$ for $i=0,1,2$ which act to cyclically rotate the basis vectors $\{1,w,w^2\}$. This yields the following $6$ matrices:
\begin{align*}
l_{1}&=\begin{bmatrix}1 & 0&0 \\0 & 1 &0 \\0&0&1\end{bmatrix} &\ \ \ \ \  e_{0}&=\begin{bmatrix}1 & 0&0 \\0 & 0 &0 \\0&0&0\end{bmatrix}\\
l_{w}&=\begin{bmatrix}0& 0&a \\1 & 0 &0 \\0&1&0\end{bmatrix} &\ \ \ \ \   e_{1}&=\begin{bmatrix}0 & 0&0 \\0 & 1 &0 \\0&0&0\end{bmatrix}\\
l_{w^2}&=\begin{bmatrix}0 & a&0 \\0 & 0 &a \\1&0&0\end{bmatrix} &\ \ \ \ \   e_{2}&=\begin{bmatrix}0 & 0&0 \\0 & 0 &0 \\0&0&1\end{bmatrix} \\
\end{align*}
which when multiplied in pairs $\{l_{w^j}e_{i}\}$ yield nine matrices
\begin{align*}
\bigg\{&\begin{bmatrix}1 & 0&0 \\0 & 0 &0 \\0&0&0\end{bmatrix},\begin{bmatrix}0 & 0&0 \\0 & 1 &0 \\0&0&0\end{bmatrix},\begin{bmatrix}0 & 0&0 \\0 & 0 &0 \\0&0&1\end{bmatrix},\\
&\begin{bmatrix}0& 0&0 \\1 & 0 &0 \\0&0&0\end{bmatrix},\begin{bmatrix}0& 0&0 \\0 & 0 &0 \\0&1&0\end{bmatrix}, \begin{bmatrix}0& 0&a \\0 & 0 &0 \\0&0&0\end{bmatrix},\\
&\begin{bmatrix}0 & 0&0 \\0 & 0 &0 \\1&0&0\end{bmatrix},\begin{bmatrix}0 & a&0 \\0 & 0 &0 \\0&0&0\end{bmatrix},\begin{bmatrix}0 & 0&0 \\0 & 0 &a \\0&0&0\end{bmatrix},\bigg\}  \\
\end{align*}
corresponding to the $\{w_n^j\# e_i\}$. Note also that this set is clearly a basis for the endomorphism ring since $\sum_{j,i} c_{j,i} l_{w^j}e_i$ equals
$$
\begin{bmatrix} c_{0,0} & c_{2,1}a & c_{1,2}a\\c_{1,0} & c_{0,1} & c_{2,2}a\\c_{2,0} & c_{1,1} & c_{0,2}\end{bmatrix}
$$
which, given that $a\in\mathbb{Q}$, gives every $3\times 3$ matrix over $\mathbb{Q}$ for unique choices of $\{c_{j,i}\}$. One sees the same motif for larger $p$ and $n$, namely a $p^n\times p^n$ matrix where every entry above the main diagonal is multiplied by $a$.
\section{Profinite Forms}
In this section we shall construct a profinite Hopf algebra form that satisfies a generalization 
of the following:
\begin{theorem}{\cite[Theorem 5]{HP}}
Let $G$ be a finitely generated group with finite automorphism group $F=Aut(G)$. Then there is a bijection
between {\bf Gal}$(k,F)$ (extensions of $k$ with Galois group $F$) and {\bf Hopf}$(k[G])$ (Hopf algebra forms of $k[G]$) which associates with each $F$-Galois extension $K$ of $k$ the Hopf
algebra
$$
H = \{ \sum c_g g \in KG | \sum f(c_g) f(g) = \sum c_g g \ for \ all \ f\in F \}
$$
Furthermore, $H$ is a $K$-form of $k[G]$ by the isomorphism
$$
\omega : H\otimes K \cong KG , \ \omega (h\otimes a)=ah
$$ 
\end{theorem}
By construction, all the $H_{n}$ are $\mathbb{Q}$-Hopf algebras which are $Q(\zeta_{n})$-forms of the group rings $\mathbb{Q}[N_{n}]$ and are examples of the above theorem in action. The reason for this is that $\Delta_{n}$ is isomorphic to the automorphism group of the cyclic group $N_{n}$ as well as to $Gal(\mathbb{Q}(\zeta_{n})/\mathbb{Q})$. What we would like to do now is to consider a profinite version of the above result. The usage of the term profinite is motivated by looking at the construction of the Galois group of a direct limit (union) of field extensions. In particular, for a base field $F$, if $L=\underset{\rightarrow}{lim} K$ where $K$ is a chain of sub-fields of $L$ containing $F$, then if each $K$ is a Galois extension of $F$ then $Gal(L/F)=\underset{\leftarrow}{lim}Gal(K/F)$ the inverse limit of the Galois groups of each of the $K/F$.\par
Here we shall consider the fields $\mathbb{Q}(w_{n})$ where each $w_{n}$ is chosen to be a $p^n$-th root of a fixed $a\in\mathbb{Q}$ which is not already a $p^n$-th root of a rational for any $n$. Even though these are not normal extensions of $\mathbb{Q}$, by what we have already shown each {\it is} Hopf-Galois over $\mathbb{Q}$ with respect to the Hopf algebras $H_{n}$. As such, we will start with an inverse system using the $H_n$. The resulting Hopf algebra will be a form of a topologically finitely generated group whose automorphism group is not finite, but which satisfies the above theorem. That the automorphism group is infinite contrasts with the setup in \cite{HP}.\par
One issue to be dealt with first is that, while the $H_n$ are all $\mathbb{Q}$-Hopf algebras, the group rings $\mathbb{Q}(\zeta_n)[N_n]$ (which contain each $H_n$) are $\mathbb{Q}(\zeta_n)$-Hopf algebras for each $n$. As such, one cannot start with a directed system involving these group rings, and then descend since these all lie in distinct categories of Hopf algebras, one for each ground field $\mathbb{Q}(\zeta_n)$.\par
As each $H_n$ is $\mathbb{Q}(\zeta_n)$-form of $\mathbb{Q}[N_n]$ then one may base change all $H_n$ up to $\mathbb{Q}(\zeta_{\infty})$ to yield group rings (and $\mathbb{Q}(\zeta_{\infty})$-Hopf algebras) $\mathbb{Q}(\zeta_{\infty})[N_n]$, as diagrammed below.
$$
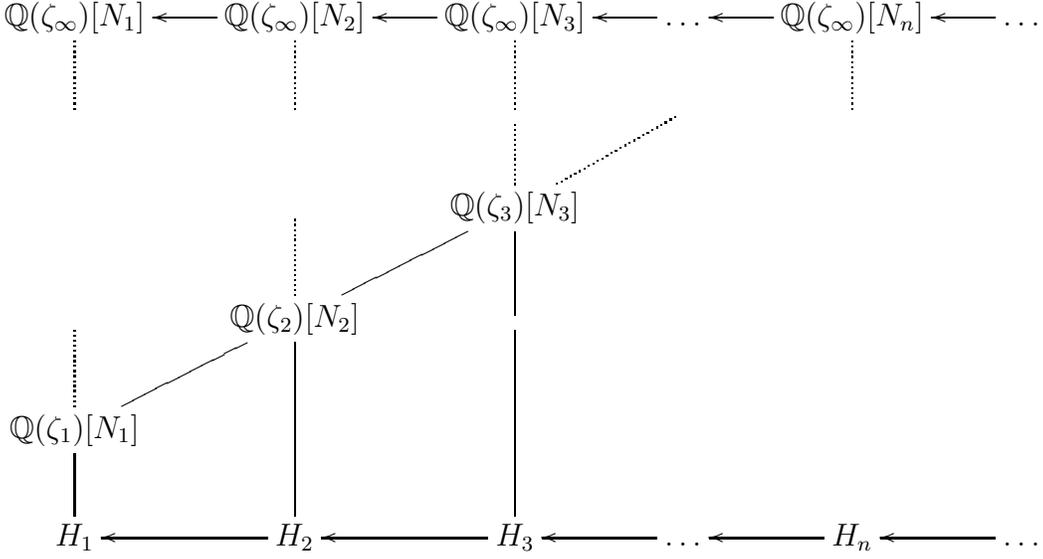
\diagram
\mathbb{Q}(\zeta_{\infty})[N_1]\ddotted  & \mathbb{Q}(\zeta_{\infty})[N_2]\lto\ddotted       &   \mathbb{Q}(\zeta_{\infty})[N_3]\lto \ddotted  &  \dots\lto & \mathbb{Q}(\zeta_{\infty})[N_n]\lto\ddotted & \lto\dots \\ 
& & & & &\\
                              &                                       & \mathbb{Q}(\zeta_{3})[N_3]\dlline\dline\udotted\urdotted & \\
                              & \mathbb{Q}(\zeta_{2})[N_2]\dlline\ddline\udotted & \ddline \\
\mathbb{Q}(\zeta_{1})[N_1]\dline\udotted & \\
                      H_{1}   &     H_{2}        \lto                  & H_{3}\lto                             & \dots\lto & H_{n}\lto                        & \lto\dots \\ 
\enddiagram$$
As such, we will define a pair of inverse systems of Hopf algebras over $\mathbb{Q}$ and $\mathbb{Q}(\zeta_{\infty})$ which will be related by descent.\par
Define $\nu_{j,i}:\mathbb{Q}(\zeta_{\infty})[N_{j}]\longrightarrow\mathbb{Q}(\zeta_{\infty})[N_{i}] \text{ for } j\geq i$ as follows:
\begin{align*}
\nu_{j,i}(q) &= q \ for\  q\in\mathbb{Q}(\zeta_{\infty}) \\
\nu_{j,i}(\sg_{j}) &= \sg_{i} 
\end{align*}
Hence $\nu_{i,i}$ is the identity map on $\mathbb{Q}(\zeta_{\infty})[N_{i}]$ and $\nu_{j,i}\circ\nu_{k,j}=\nu_{k,i} \ for \ k \geq j \geq i$ and we have the following obvious fact.
\begin{lemma}
\label{nu}
$\nu_{j,i}$ is a surjective map of $\mathbb{Q}(\zeta_{\infty})$-Hopf algebras.
\end{lemma}
\begin{proof}
The surjectivity is obvious given that $\nu_{j,i}$ is surjective as a group homomorphism from $N_j$ to $N_i$ which, since it acts as the identity on the coefficients, is also seen to be a Hopf algebra morphism between the respective group rings.
\end{proof}
We also need to consider whether the $\nu_{j,i}$ restrict to the $H_n$. As each $H_n\subseteq \mathbb{Q}(\zeta_{n})[N_n]\subseteq \mathbb{Q}(\zeta_{\infty})[N_n]$ is the span of $\{e_{n,i}\}$ given in \ref{Hbasis} then it makes sense to compute $\nu_{n,n-1}(e_{n,i})$.
\begin{lemma}
\label{nuhopf}
For $e_{n,i}$ where $i\in\mathbb{Z}_{p^n}$ as given in \ref{Hbasis}, then 
$$\nu_{n,n-1}(e_{n,i})=\begin{cases} e_{n-1,i/p}\ \ \text{if }i\in p\mathbb{Z}_{p^{n-1}}\subseteq\mathbb{Z}_{p^n} \\ 0\ \ \text{otherwise}\end{cases}$$
\end{lemma}
\begin{proof}
We have 
$$e_{n,i}=\frac{1}{p^n}\sum_{j=0}^{p^n-1}\zeta_n^{-ij}\sg_n^{j}$$
so that, if $i\in p\mathbb{Z}_{p^{n-1}}$ then
\begin{align*}
\nu_{n,n-1}(e_{n,i})&=\frac{1}{p^n}\sum_{j=0}^{p^n-1}\zeta_n^{-ij}\sg_{n-1}^{j}\\
                  &=\frac{1}{p^n}\sum_{j=0}^{p^n-1}\zeta_{n-1}^{-\frac{i}{p}j}\sg_{n-1}^{j}\\
                  &=\frac{p}{p^n}\left(\sum_{j=0}^{p^{n-1}-1}\zeta_{n-1}^{-\frac{i}{p}j}\sg_{n-1}^{j}\right)\\
                  &=e_{n-1,i/p}\\
\end{align*}
where the passage from $\zeta_n$ to $\zeta_{n-1}$ is due to the fact that $i$ is a multiple of $p$. Since each $j\in\mathbb{Z}_{p^n}$ can be written as $ap^{n-1}+b$ where $a\in\mathbb{Z}_{p}$ and $b\in\mathbb{Z}_{p^{n-1}}$, then if $i$ is not a multiple of $p$ we have
\begin{align*}
\nu_{n,n-1}(e_{n,i})&=\frac{1}{p^n}\sum_{j=0}^{p^n-1}\zeta_n^{-ij}\sg_{n-1}^{j}\\
                  &=\frac{1}{p^n}\sum_{b=0}^{p^{n-1}-1}  \sum_{a=0}^{p-1}\zeta_n^{-i(ap^{n-1}+b)}\sg_{n-1}^{ap^{n-1}+b}\\
                  &=\frac{1}{p^n}\sum_{b=0}^{p^{n-1}-1} \sum_{a=0}^{p-1}\zeta_n^{-i(ap^{n-1}+b)}\sg_{n-1}^{b}\\
                  &=\frac{1}{p^n}\sum_{b=0}^{p^{n-1}-1}\zeta_n^{-ib}\left( \sum_{a=0}^{p-1}\zeta_n^{-i(ap^{n-1})}\right)\sg_{n-1}^{b}\\
                  &=\frac{1}{p^n}\sum_{b=0}^{p^{n-1}-1}\zeta_n^{-ib}\left( \sum_{a=0}^{p-1}\zeta_1^{-ia}\right)\sg_{n-1}^{b}\\
                  &=\frac{1}{p^n}\sum_{b=0}^{p^{n-1}-1}\zeta_n^{-ib}\left( 0\right)\sg_{n-1}^{b}\\
                  &=0.\\
\end{align*}
\end{proof}
It is interesting to note that, $\nu_{n,n-1}:H_n\rightarrow H_{n-1}$ where $H_n=(\mathbb{Q}N_n)^{*}$ and $H_{n-1}=(\mathbb{Q}N_{n-1})^{*}$ can be viewed as the dual of the natural map $\alpha_{n-1,n}:\mathbb{Q}N_{n-1}\rightarrow \mathbb{Q}N_n$ given by $\alpha_{n-1,n}(\sigma_{n-1})=\sigma_{n}^p$ since then $\alpha_{n-1,n}^{*}$ would be defined by $\alpha_{n-1,n}^{*}(e_{n,i})(\sigma_{n-1}^j)=e_{n,i}(\sigma_{n}^{pj})=\delta_{i,pj}$. As such $\alpha_{n-1,n}^{*}(e_{n,i})=0$ if $i$ is not a multiple of $p$, and if $i$ were a multiple of $p$ then $\alpha_{n-1,n}^{*}(e_{n,i})=e_{n-1,i/p}$ which is exactly what we get with $\nu_{n,n-1}$.\par
The $H_n$ are constructed as $(\mathbb{Q}(\zeta_n)[N_n])^{\Delta_n}$ where $\Delta_n$ acts diagonally on the scalars and group elements by virtue of it being $Gal(\mathbb{Q}(\zeta_n)/\mathbb{Q})$ and isomorphic to $Aut(N_n)$. In a related way we will consider the action of $Gal(\mathbb{Q}(\zeta_{\infty})/\mathbb{Q})$ on each $\mathbb{Q}(\zeta_{\infty})[N_n]$.
\noindent Define $\phi_{j,i}$ : $\Delta_{j}\longrightarrow\Delta_{i} \ (j\geq i)$ by $\phi_{j,i}(\dl_{j})=\dl_{i}$.
It is easy to verify that $\{\Delta_{i},\phi_{j,i}\}$ is also an inverse system and we shall define
$$\DI = \underleftarrow{lim}\Delta_{i}$$
which, amongst other things, is the Galois group of the profinite extension $\mathbb{Q}(\zeta_{\infty})/\mathbb{Q}$. Furthermore, if we restrict $\nu_{j,i}$ to the $N_{j}$ then it is clear that
$\{ N_{j},\nu_{j,i} \}$ is an inverse system and we shall define $\NI=\underleftarrow{lim}N_{j}$. Each $N_{j}$ is cyclic of order $p^j$ and $\DI$ is also the inverse limit of the automorphism groups of each $N_{j}$. Since a given primitive root $\pi$ mod $p$ is also a primitive root mod $p^n$ then we can choose the Galois group of $\mathbb{Q}(\zeta_{j})$ to be generated by an element which acts to a raise $\zeta_{j}$ to $\pi$ for all $j$. Similarly, each automorphism group is generated by an element which acts to raise $\sigma_n$ to the same power as well. We have the following which is known, for example \cite[p.656]{LF}, but we present here for use later.
\begin{proposition}
\begin{align*}
\NI&\cong\{\{\sg_{j}^{a_j}\}\in\Pi_{j=1}^{\infty} N_{j}\ |\ \nu_{j,i}(\sg_{j}^{a_j})=\sg_{i}^{a_i}\ (mod\ p^i)  )\} \\
   &\cong\{\{\sg_{j}^{a_j}\}\in\Pi_{j=1}^{\infty} N_{j}\ |\ a_j\equiv a_i\ (mod\ p^i)  )\} \\
   &\cong J_p\text{ the p-adic integers}\\
\DI&\cong\{\{\dl_{i}^{e_{i}}\}\in\Pi_{j=1}\Delta_{j}\ |\ \phi_{j,i}(\dl_{j}^{e_{j}})=\dl_{i}^{e_{i}}\}\\
   &\cong\{\{\dl_{i}^{e_{i}}\}\in\Pi_{j=1}\Delta_{j}\ |\ e_j\equiv e_i\ (mod\ p^i) \}\\
   &\cong (J_p)^{*}\text{ the unit p-adic integers}\\
\end{align*}
\end{proposition}
Note, exponents $a_j$ in the definition of $\NI$ lie in $\mathbb{Z}_{p}$ whereas the $e_j$ in the definition of $\DI$ lie in $(\mathbb{Z}_{p})^{*}$. And since component-wise $\Delta_{j}$ is the automorphism group of each $N_{j}$, then the congruence conditions on the respective exponents $a_j$ and $e_j$ yield the following which is also known. 
\begin{proposition} $Aut(\NI)\cong\DI$ 
\label{aut}
\end{proposition}
This could also be deduced from the fact that $Aut(\mathbb{Z}_p)\cong (\mathbb{Z}_{p})^{*}$. Moreover, this implies that $\DI$ acts by restriction on each $N_i$ as $Aut(N_i)$. As such, $\DI$ acts on each $\mathbb{Q}(\zeta_{\infty})[N_i]$ which yields the following.
\begin{lemma} 
\label{commute}
The following diagram commutes:
$$\diagram
\mathbb{Q}(\zeta_{\infty})[N_{j}] \dto_{\dl_{j}}\rto^{\nu_{j,i}}  & \mathbb{Q}(\zeta_{\infty})[N_{i}] \dto^{\dl_{i}} \\
\mathbb{Q}(\zeta_{\infty})[N_{j}] \rto_{\nu_{j,i}}                     & \mathbb{Q}(\zeta_{\infty})[N_{i}] \\
\enddiagram$$
\end{lemma}

\noindent Lemmas \ref{nu},\ref{nuhopf}, and \ref{commute} imply that
$$\{\mathbb{Q}(\zeta_{\infty})[N_{j}]\} \ and \ \{(\mathbb{Q}(\zeta_{\infty})[N_j])^{\DI}\}\ and\ \{H_j\}$$
are inverse systems with respect to $\nu_{j,i}$ where we have:
$$\underleftarrow{lim}\mathbb{Q}(\zeta_{\infty})[N_{j}] = \QI[\NI]$$
and if we define 
$$\HI = \underleftarrow{lim} H_{j}$$ 
we ask what the relationship is between $\QI[\NI]$ and $\HI$?\par
\noindent We have the following:
\begin{theorem}
\begin{align*}
&(a) \ \TI\HI \cong \QI[\NI] \\
&(b) \ \HI = (\QI[\NI])^{\DI}
\end{align*}
\end{theorem}
\begin{proof} 
That $\DI$ acts on $\QI[\NI]$ is clear given the previous observations that $\DI$ is isomorphic to  $Gal(\QI/\mathbb{Q})$ and $Aut(\NI)$. Moreover by \ref{commute} we have that $\dl_{i}(\nu_{j,i}(x))=\nu_{j,i}(\dl_{j}(x)) = \nu_{j,i}(\phi_{j,i}(\dl_{i})(x))$ for all $x\in\mathbb{Q}(\zeta_{j})[N_{j}]$.  (i.e. we may think of the $\nu_{j,i}$'s as $\DI$-maps) 
By virtue of \ref{nu} and \ref{commute} we have:
$$\diagram
H_{1}&  H_{2}\lto^{\nu_{2,1}} & \dots\lto^{\nu_{3,2}} & H_{n}\lto & \lto\dots \\ 
                                   & \HI\ulto^{\psi_1}\uto^{\psi_2}\urrto^{\psi_n} \\
\enddiagram$$
where the $\psi_i$ are the canonical projections out of the direct limit. If we base change the above up to $\QI$ then we have the following:

$$\diagram
\QI\otimes H_{1} &  \QI\otimes H_{2}\lto^{\nu_{2,1}\otimes 1} & \dots\lto^{\ \ \ \ \ \ \nu_{3,2}\otimes 1} & \QI\otimes H_{n}\lto & \lto\dots \\ 
		                   & \QI\otimes \HI\ulto^{\psi_1\otimes 1}\uto^{\psi_2\otimes 1}\urrto^{\psi_n\otimes 1} \\
\enddiagram$$

But since $\mathbb{Q}(\zeta_{n})\otimes_{\mathbb{Q}}H_{n}\cong\mathbb{Q}(\zeta_{n})[N_{n}]$ and since $\mathbb{Q}(\zeta_{n})\subseteq\QI$ for all $n$ then the above diagram becomes:
$$\diagram
\QI N_{1} &  \QI N_{2}\lto^{\nu_{2,1}\otimes 1} & \dots\lto^{\ \ \nu_{3,2}\otimes 1} & \QI N_{n}\lto & \lto\dots \\ 
		                   & \QI\otimes\HI\ulto^{\psi_1\otimes 1}\uto^{\psi_2\otimes 1}\urrto^{\psi_n\otimes 1} \\
\enddiagram$$
In general, direct limits 'commute' with the taking of tensor products, but the same is not true generally for inverse limits, since tensor product does not usually commute with direct products. However, we can 'build up' to $\QI\otimes\HI$ by first looking at $\mathbb{Q}(\zeta_m)\otimes \Pi_{n\geq 1}H_n$ where each $\mathbb{Q}(\zeta_m)$ is certainly finitely generated and projective as a $\mathbb{Q}$-module. As such, by \cite[Prop. 1.1]{Webb} the canonical map $\mathbb{Q}(\zeta_m)\otimes \Pi_{n\geq 1}H_n\rightarrow \Pi_{n\geq 1}\mathbb{Q}(\zeta_m)\otimes H_n$ is a bijection. By \ref{QmtensorHn} we have that $\mathbb{Q}\otimes H_n$ contains $\mathbb{Q}(\zeta_m)[\langle \sigma_n^{p^{n-m}}\rangle]$. And since tensor product {\it does} commute with direct limits, we have
\begin{align*}
\QI\otimes \Pi_{n\geq 1}H_n&\cong (\mathop{\lim_{\longrightarrow}}_{m} \mathbb{Q}(\zeta_m))\otimes \Pi_{n\geq 1}H_n\\
                        &\cong \mathop{\lim_{\longrightarrow}}_{m} (\mathbb{Q}(\zeta_m)\otimes \Pi_{n\geq 1}H_n)\\
                        &\cong \mathop{\lim_{\longrightarrow}}_{m} (\Pi_{n\geq 1}\mathbb{Q}(\zeta_m)\otimes H_n)\\
\end{align*}
where now, viewing the direct limit as union, we have that each component in the direct product $\QI\otimes\Pi_{n\geq 1}H_n$ is exactly $\QI[N_n]$. So, within this direct product, we have the sub-algebra determined by the $\nu_{n,n-1}$ which is the inverse limit of $\QI[N_n]$, that is
$$
\TI\HI=\TI(\underleftarrow{lim}H_{n})=\underleftarrow{lim}(\QI[N_{n}])=\QI[\NI]
$$
which completes the proof of (a).\par
To show (b) we we shall use the canonical constructions of $\underleftarrow{lim}(\mathbb{Q}(\zeta_{\infty}) N_{j})=\QI[\NI]$ and $\underleftarrow{lim}H_{i}=\HI$. We have
\begin{align*}
\QI[\NI]&\cong\{\{\gamma_{j}\}\in\Pi_{j=1}^{\infty} \mathbb{Q}(\zeta_{\infty})[N_{j}]\ |\ \nu_{j,i}(\gamma_j)=\gamma_i  )\} \\
\HI&\cong\{\{\gamma_{j}\}\in\Pi_{j=1}^{\infty} H_{j}\ |\ \nu_{j,i}(\gamma_j)=\gamma_i  )\} \\
\end{align*}
where the usage of $\gamma_j$ in both is not an abuse of notation since $H_{j}\subseteq \mathbb{Q}(\zeta_{\infty})[N_{j}]$ for each $j$ and so also there is containment of the direct products. \par
Now if $\hat{\delta}=\{\delta_{j}^{e_j}\}\in\DI$ and $\{\gamma_j\}\in \QI[\NI]$ then $\{\gamma_j\}^{\hat{\delta}}=\{\gamma_j\}$ implies that $\delta^{e_j}(\gamma_j)=\gamma_j$ for all $j$. The question is does this imply that $\gamma_j\in H_j$ for all $j\geq 1$? Yes, because $\DI$ contains $\{\delta_j^{e_j}\}$ where $e_j=1$ for any specified $j\geq 1$, so indeed $\delta_j(\gamma_j)=\gamma_j$ for each $j\geq 1$ and therefore $\{\gamma_j\}\in\Pi_{j=1}^{\infty}H_j$. But now, since $\{\gamma_j\}\in\QI[\NI]$ we have $\nu_{j,i}(\gamma_j)=\gamma_i$ so when restricted to $\gamma_j\in H_j$ we have $\{\gamma_j\}\in \HI$. Thus $(\QI[\NI])^{\DI}\subseteq \HI$.\par
 The other inclusion is obvious since $\HI\subseteq\QI[\NI]$ and is fixed by {\it all} of $\Pi_{j=1}^{\infty}\Delta_j$, so therefore by $\DI$.
\end{proof}
One should note that $\HI$ is generated by $\{e_{n,i_n}\}$ (in the direct product) where $i_n=p\cdot i_{n-1}$ which makes sense if we go back to the observation earlier that each $H_{n}$ is isomorphic to $(\mathbb{Q}N_{n})^{*}$ where now
\begin{align*}
\HI &\cong \underleftarrow{lim}(\mathbb{Q}[N_{n}])^{*}\\
    &=\underleftarrow{lim}Hom(\mathbb{Q}[N_{n}],\mathbb{Q})\\
    &\cong Hom(\underrightarrow{lim}\mathbb{Q}[N_{n}],\mathbb{Q})\\
\end{align*}
where $\underrightarrow{lim}\mathbb{Q}[N_{n}]$ is the group ring over the $p$-Pr\"ufer group formed from the union of the $\{N_{n}\}$ since each are cyclic of order $p^n$. It is also interesting to note that (again as cited in \cite{LF}) the automorphism group of the $p$-Pr\"ufer group is also isomorphic to $\DI\cong (J_p)^{*}$. Undoubtedly, this is obliquely a manifestation/result of the fact (cited in \cite{Armacost} for example) that $J_p$ and the $p$-Pr\"ufer group are Pontryagin duals of each other. That $\HI$ is a Hopf algebra is a consequence of the fact that $\QI/\mathbb{Q}$ (being a direct limit of faithfully flat extensions) is a faithfully flat extension so by general descent theory \cite{P} such an extension preserves and reflects structures such as Hopf algebras. That is, since $\HI$ is a $\QI$ form of $\mathbb{Q}[\NI]$ which is a Hopf algebra then $\HI$ is a Hopf algebra over $\mathbb{Q}$. This is also a rare example of where the dual of an infinite dimensional Hopf algebra is itself a Hopf algebra since typically one loses the 'closure' of the induced co-algebra structure on the dual in this setting. Aside from this descent theoretic proof of this fact, it is the author's conjecture that $\HI$ is Hopf, even though it is the dual of an infinite group, since said infinite group is torsion. That is the finite dual {\it is} the dual.\par
Additionally, in view of \ref{aut} and that $Gal(\QI/\mathbb{Q})=\DI$, this yields a nice generalization of Theorem 5 of \cite{HP}.\par
\section{$\QWI / \mathbb{Q}$ as an $\HI$-Galois extension}
\noindent The field $\QWI$ is certainly linearly disjoint to $\QI$ over $\mathbb{Q}$. Moreover, the normal closure of $\QWI$ contains $\QI\QWI$ since the splitting field for all polynomials of the form $x^{p^n}-a$ must contain $\QI$ and $\QWI$. It then must be contained in $\QI\QWI$ since the minimal polynomial of any element in $\QWI$ is split in $\QI\QWI$. The question is, can we view $\QWI/\mathbb{Q}$ as a Hopf-Galois extension with respect to the action of $\HI$?\par
The difficulty that arises is in verifying that $\HI$ acts on $\QWI$ in the same way that a profinite Galois group would act on a direct limit (union) of the intermediate fields. Since $\QWI$ is the direct limit (union) over all $\mathbb{Q}(w^{p^n})\subseteq \QWI$ then if it were a normal extension, its Galois group would be the inverse limit of the Galois groups of each intermediate field over $\mathbb{Q}$. We have that for each $n$ there is an isomorphism $\mathbb{Q}(w_n)\# H_{n}\cong End_{\mathbb{Q}}(\mathbb{Q}(w_n))$ and since $\QWI=\underrightarrow{lim}\mathbb{Q}(w_n)$ then we wish to examine the relationships between $End_{\mathbb{Q}}(\QWI)$ and $\QWI\#\HI$.\par
We first observe that
\begin{align*}
 End(\QWI)&=Hom( \mathop{\lim_{\longrightarrow}}_{n}\mathbb{Q}(w_n), \mathop{\lim_{\longrightarrow}}_{m}\mathbb{Q}(w_m))\\
          &\cong  \mathop{\lim_{\longleftarrow}}_{n} Hom(\mathbb{Q}(w_n), \mathop{\lim_{\longrightarrow}}_{m}\mathbb{Q}(w_m))\\ 
          &\cong \mathop{\lim_{\longleftarrow}}_{n} \left[\mathop{\lim_{\longrightarrow}}_{m} Hom(\mathbb{Q}(w_n),\mathbb{Q}(w_m))\right]\\ 
          &
\end{align*} 
where the direct limit (over $n$) in the first component becomes the inverse limit induced by the natural restriction map 
$$Hom(\mathbb{Q}(w_{n}),\mathop{\lim_{\longrightarrow}}_{m} \mathbb{Q}(w_m))\longrightarrow Hom(\mathbb{Q}(w_{n-1},\mathop{\lim_{\longrightarrow}}_{m}\mathbb{Q}(w_m))$$ since $\mathbb{Q}(w_{n-1})\subseteq \mathbb{Q}(w_{n})$. The direct limit (over $m$) from the second component is permitted to be moved outside due to the fact that $\mathbb{Q}(w_{n})$ is finitely presented for each $n$.
\begin{proposition} 
The algebra $Hom(\mathbb{Q}(w_n),\mathbb{Q}(w_m))$ is isomorphic to\par
(a) $\mathbb{Q}(w_{m})\# \overline{H}_{m,n}$ where $\overline{H}_{m,n}$ is the sub-algebra of $H_{m}$ spanned by $\{e_{m,i}\}$ for $i\in p^{m-n}\mathbb{Z}_{p^n}\subseteq\mathbb{Z}_{p^m}$ if $m\geq n$ or\par
(b) the sub-algebra of $\mathbb{Q}(w_{n})\# H_{n}$ spanned by $\{w_n^j\#e_{n,i}\}$ where $i\in\mathbb{Z}_{p^n}$ where $p^{n-m}\div j+i$ if $m<n$.
\end{proposition}
\begin{proof} 
If $m\geq n$ then $Hom(\mathbb{Q}(w_n),\mathbb{Q}(w_m))\subseteq  Hom(\mathbb{Q}(w_m),\mathbb{Q}(w_m))$ where the latter is isomorphic to $\mathbb{Q}(w_{m})\# H_m$. As given in \eqref{endact}, 
$$
(w_m^j\# e_{m,i})(w_{m}^{tp^{m-n}})=\begin{cases} 0\ \ \ \ \ \ \ \ \ \ \ \ \ i\neq tp^{m-n} \\ w_m^{j+tp^{m-n}}\ i=tp^{m-n}\end{cases}
$$
so those elements of $\mathbb{Q}(w_{m})\# H_m$ that have domain $\mathbb{Q}(w_n)=\mathbb{Q}((w_m)^{p^{m-n}})$ are exactly $w_{m}^j\# e_{m,i}$ for $i\in p^{m-n}\mathbb{Z}_{p^n}\subseteq\mathbb{Z}_{p^m}$. If one views these as $p^m\times p^m$ matrices, then this sub-algebra consists of those matrices where (if numbering columns from $0$) have non-zero columns if the column index is in $p^{m-n}\mathbb{Z}_{p^n}\subseteq\mathbb{Z}_{p^m}$.\par
If $m<n$ then $w_m=w_n^{p^{n-m}}$ and so $Hom(\mathbb{Q}(w_n),\mathbb{Q}(w_m))\subseteq  Hom(\mathbb{Q}(w_n),\mathbb{Q}(w_n))$ where the latter is isomorphic to $\mathbb{Q}(w_n)\# H_n$. However, here the co-domain is restricted to $\mathbb{Q}(w_m)\subseteq \mathbb{Q}(w_n)$ so any $w_n^t$ must map to $\mathbb{Q}(w_m)$. Again, by \eqref{endact}, 
$$
(w_n^j\# e_{n,i})(w_{n}^{t})=\begin{cases} 0\ \ \ \ \ \ \ \ \ \ \ \ \ i\neq t \\ w_n^{j+t}\ \ \ \ \ \ \ i=t\end{cases}
$$
so as $t$ varies over $\mathbb{Z}_{p^n}$ so must $i$ which means $j$ is restricted by the condition that $j+t$ must be a multiple of $p^{n-m}$ which means $p^n\cdot p^{m}$ choices for $(j,i)$ which is dimensionally correct given the domain and co-domain of the homomorphisms in question.
\end{proof}
If one now considers the direct limit
$$\underset{m}{\underset{\longrightarrow}{\lim}} Hom(\mathbb{Q}(w_n),\mathbb{Q}(w_m))$$
for a given $n$, then one is looking at endomorphisms of $\mathbb{Q}(w_m)$, generated by left multiplication by elements of $\mathbb{Q}(w_m)$ together with those arising from each sub-algebra, either $H_n$ (when $m<n$) or $\overline{H}_{m,n}$ for those $m\geq n$. If for notational uniformity we define $\overline{H}_{m,n}=H_n$ when $m<n$ then we wish to first consider the direct limit  $\underset{m}{\underset{\longrightarrow}{\lim}}\overline{H}_{m,n}$.\par
Although we are considering the action on $\QWI$ by $\HI$, which is the {\it inverse} limit of the $H_{m}$, there is a natural embedding of $H_{m}$ into $H_{m+1}$ via $e_{m,i}\mapsto e_{m+1,pi}$. Concordantly, for $m\geq n$ this restricts to an embedding $\overline{H}_{m,n}\hookrightarrow \overline{H}_{m+1,n}$ for each $n$ since if $i\in p^{m-n}\mathbb{Z}_{p^n}$ then $pi\in p^{m+1-n}\mathbb{Z}_{p^n}$. However, this embedding is, in fact, an isomorphism since $dim(\overline{H}_{m,n})=p^n$ for each $m$! Moreover, for each $m<n$, $\overline{H}_{m,n}=H_{n}$ so that, in fact:
$$\underset{m}{\underset{\longrightarrow}{\lim}}\overline{H}_{m,n}\cong H_n$$
and since the union of the scalars $\mathbb{Q}(w_m)$ is $\QWI$ then we have proved:
\begin{proposition}
$$\underset{m}{\underset{\longrightarrow}{\lim}} Hom(\mathbb{Q}(w_n),\mathbb{Q}(w_m))\cong \QWI\#H_n$$ 
\end{proposition}
\noindent This leads us to the main result for this section.
\begin{theorem}
$$
End(\QWI,\QWI)\cong \QWI\#\HI
$$
\end{theorem}
\begin{proof}
The principal observation needed is that the inverse limit
\begin{align*}
&\mathop{\lim_{\longleftarrow}}_{n} Hom(\mathbb{Q}(w_n),\QWI)\\
&\cong \mathop{\lim_{\longleftarrow}}_{n} \QWI\# H_n\\
\end{align*}
arises from the natural restriction maps, but these correspond exactly to the $\nu_{n,n-1}$ given earlier in the construction of $\HI$ which act as the identity on $\QWI$, that is:
$$
\mathop{\lim_{\longleftarrow}}_{n} \QWI\#\overline{H}_n\cong \QWI\#\HI
$$
so that the endomorphism ring of $\QWI$ is the latter smash product, making $\QWI/\mathbb{Q}$ a Hopf-Galois extension with respect to the action of $\HI$.
\end{proof}
The last consideration is if $\HI$ can be viewed within the Greither-Pareigis theory. We have that $\HI$ is a $\QI$-form (and therefore a $\QWII$-form) of the group ring $\mathbb{Q}\NI$. In terms of normal complements involving the Galois groups of the relevant intermediate extensions, namely
\begin{align*}
\NI&=Gal(\QWII/\QI)\\
\DI&=Gal(\QI/\mathbb{Q})\\
\NI\DI&=Gal(\QWI\QI/\mathbb{Q}\\
\end{align*}
the extension $\QWI / \mathbb{Q}$ is almost classical. (i.e. '$N$' is $\NI$) The delicate part is if $\NI$ can be viewed as a regular subgroup, and moreover, of what ambient symmetric group? The construction of $\HI$ parallels that of a profinite Galois group acting on a direct limit (union) of field extensions, where the restriction to a given sub-field in the chain corresponds to the action of the Galois group acting on that field extension. Here, $\HI$ acts by restriction on $\mathbb{Q}(w_n)$ as $H_n$ where, by Greither-Pareigis, there is a corresponding regular subgroup of $N_n\leq Perm(\Gamma_n/\Delta_n)$. Observe however that, as seen earlier, $\Gamma_n=N_n\Delta_n$ so that $\Gamma_n/\Delta_n=\{\sigma_n^i\Delta_n\}$ and where $N_n$ acts naturally on the left, just as it would act on itself via the left regular representation. (i.e. identify $Perm(\Gamma_n/\Delta_n)\cong Perm(N_n)$) In the limit, the analogue would be $\NI$ a regular subgroup of $Perm(\NI\DI/\DI)\cong Perm(\NI)$, via the left action, given that the left regular representation is the canonical example of a regular permutation group.\par
As such, in any related construction of an inverse limit of Hopf algebras acting on intermediate extensions, we should expect the restriction to any intermediate extension to also give rise to a regular subgroup embedded in the corresponding ambient symmetric group. And for the resulting Hopf algebra to be a form of a group ring over a profinite group, similarly embedded in the corresponding (infinite) ambient symmetric group.
\section{Other Radical Extensions}
As given in \ref{mine}, for a radical extension of the form $k(w)/k$, as one increases the number of $p$-th power roots of unity in the base field, the number of Hopf-Galois structures, including the number of almost classical structures increases as well. For example, $\mathbb{Q}(\zeta_1,w_n)/\mathbb{Q}(\zeta_1)$ has $p$ Hopf-Galois structures, all $p^{min(1,n-1)}=p^1=p$ of which are almost classical. Moreover, the $N$'s which arise are all cyclic of order $p^n$.
For the case of $\mathbb{Q}(\zeta_1,w_n)/\mathbb{Q}(\zeta_1)$ we have 
$$\Gamma_{n,1}=Gal(\mathbb{Q}(\zeta_n,w_n)/\mathbb{Q}(\zeta_1))=\langle\sigma_n,\beta_n\rangle$$
where $\langle\sigma_n\rangle=Gal(\mathbb{Q}(\zeta_n,w_n)/\mathbb{Q}(\zeta_n))$ which we shall denote by $N_{n,0}$, which is cyclic of order $p^n$, of course, and $\langle\beta_n\rangle=Gal(\mathbb{Q}(\zeta_n,w_n)/\mathbb{Q}(\zeta_1,w_n))$, which is cyclic of order $p^{n-1}$ We note, in passing, that $\Gamma_{n,1}$ is the Sylow $p$-subgroup of $Gal(\mathbb{Q}(\zeta_n,w_n)/\mathbb{Q})$ since $\langle\beta_n\rangle$ is the Sylow $p$-subgroup of $Gal(\mathbb{Q}(\zeta_n)/\mathbb{Q})$.

\noindent One can show (by \cite[Theorem 3.3]{Kohl1998}) that $N_{n,0}$ and $N_{n,i}=\langle\sigma^i\beta^{p^{n-2}}\rangle$ for $i\in U_p$ are the $p$ different normal complements to $\langle\beta_n\rangle$ in $\Gamma_{n,1}$, all of which are cyclic of order $p^n$ of course. If we denote by $E_{n,i}=(\mathbb{Q}(\zeta_n,w_n))^{N_i}$ then $\mathbb{Q}(\zeta_1,w_n)/\mathbb{Q}(\zeta_1)$ is Hopf-Galois with respect to the action of $H_{n,i}=(E_i[N_{n,i}])^{\langle\beta_n\rangle}$
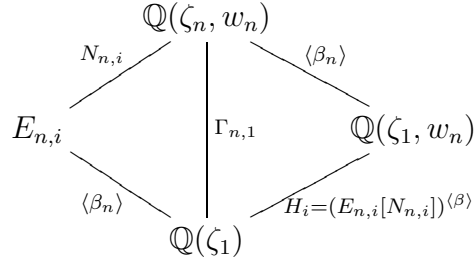
$$\diagram
& \mathbb{Q}(\zeta_{n},w_{n}) \dlline_{N_{n,i}} \ddline^{\Gamma_{n,1}} \drline^{\langle\beta_n\rangle} \\
E_{n,i} \drline_{\langle\beta_n\rangle} & &\mathbb{Q}(\zeta_1,w_{n}) \dlline^{\ \ H_i=(E_{n,i}[N_{n,i}])^{\langle\beta\rangle}} \\
&\mathbb{Q}(\zeta_1)
\enddiagram$$
To see the relationship between the $N_{n,i}$ for different $n$, consider first the relationship between the $n=2$ and $n=3$ cases. First observe that $N_{0,2}\subseteq N_{0,3}$ where, concordantly $E_{2,0}=\mathbb{Q}(\zeta_2)=(\mathbb{Q}(\zeta_2,w_2))^{N_{2,0}}\subseteq(\mathbb{Q}(\zeta_3,w_3))^{N_{3,0}}=\mathbb{Q}(\zeta_3)=E_{3,0}$. For the other $N_{2,i}$ and $N_{3,j}$ we have the following.
$$
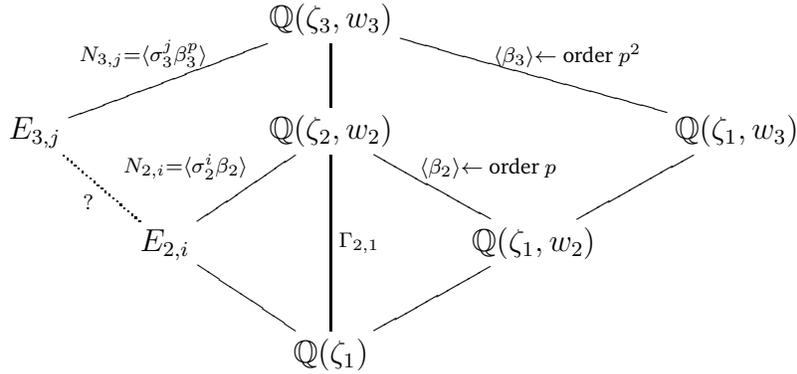
\diagram
               &                   &   \mathbb{Q}(\zeta_3,w_3)\dllline_{N_{3,j}=\langle\sigma_3^{j}\beta_3^{p}\rangle\ \ }\dline\drrline^{\ \ \langle\beta_3\rangle\leftarrow\text{ order $p^2$}}   &    & \\
 E_{3,j}\drdotted_{?}  &                   &   \mathbb{Q}(\zeta_2,w_2)\dlline_{N_{2,i}=\langle\sigma_2^{i}\beta_2\rangle\ \   }\drline^{\ \ \langle\beta_2\rangle\leftarrow\text{ order $p$}}\ddline^{\Gamma_{2,1}}   &   & \mathbb{Q}(\zeta_1,w_3)\dlline \\
               &  E_{2,i}\drline    &                                                   &    \mathbb{Q}(\zeta_1,w_2)\dlline &  \\
               &                   &             \mathbb{Q}(\zeta_1)                    &                                  &\\
\enddiagram$$
We have that $Gal(\mathbb{Q}(\zeta_3,w_3)/\mathbb{Q}(\zeta_1))=\langle\sigma_3,\beta_3\rangle$ where $\sigma_3(w_3)=\zeta_3w_3$ of course, and $\beta_3$ generates the Sylow $p$-subgroup of $Gal(\mathbb{Q}(\zeta_3,w_3)/\mathbb{Q}(w_3))$, which, by natural irrationality, is isomorphic to the Sylow $p$-subgroup of $Gal(\mathbb{Q}(\zeta_3)/\mathbb{Q})$, namely $Gal(\mathbb{Q}(\zeta_3)/\mathbb{Q}(\zeta_1))$. As such $\beta_3(\zeta_3)=\zeta_3^{\pi^{(p-1)}}$ where $\pi$ is the primitive root mod $p$, which we observed earlier is the same for all higher powers of $p$. And since $w_3^p=w_2$ and $\zeta_3^p=\zeta_2$ then $Gal(\mathbb{Q}(\zeta_2,w_2)/\mathbb{Q}(\zeta_1))$ equals $\langle\sigma_2,\beta_2\rangle$ where $\sigma_3^p=\sigma_2$. Indeed, $\sigma_3(w_2)=\sigma_3((w_3)^p)=(\zeta_3w_3)^p=\zeta_2w_2=\sigma_2(w_2)$ and similarly $\beta_3(\zeta_2)=\beta_3((\zeta_3)^p)=\zeta_3^{p\pi^{(p-1)}}=\zeta_2^{\pi^{(p-1)}}=\beta_2(\zeta_2)$, that is the action of $\sigma_3$ restricts to $\sigma_2$ and $\beta_3$ to $\beta_2$. As such, for $E_{3,j}$ and $E_{2,i}$ given above, $E_{2,i}\subseteq E_{3,j}$ only when $i=j$. Moreover, we have a natural surjection $N_{3,i}\rightarrow N_{2,i}$ given by $\sigma_3^i\beta_3^{p}\mapsto\sigma_2^i\beta_2$. In general therefore, by viewing $Gal(\mathbb{Q}(\zeta_{n-1},w_{n-1})/\mathbb{Q}(\zeta_1)\subseteq Gal(\zeta_{n},w_{n})/\mathbb{Q}(\zeta_1))$ we have the following
\begin{proposition}
For $Gal(\mathbb{Q}(\zeta_n,w_n)/\mathbb{Q}(\zeta_1))=\langle\sigma_n,\beta_n\rangle$ then for $i\in\{0,\dots,p-1\}$ there is containment $E_{n-1,i}\subseteq E_{n,i}$ where $E_{n,i}$ is the fixed field of $N_{n,i}$. Moreover $\{N_{n,i},\nu_{n,n-1}\}$ (for $n\geq 3$) forms an inverse system where for $i=0$ $\nu_{n,n-1}(\sigma_n)=\sigma_{n-1}$, and for $i\in U_p$ that $\nu_{n,n-1}(\sigma_n^i\beta_n^{p^{n-2}})=\sigma_{n-1}^i\beta_{n-1}^{p^{n-3}}$.
\end{proposition}
Also, one can observe that $\langle\beta_n\rangle$ normalizes each $N_{n,i}$ so that $N_{n,i}$ is a normal complement to $\langle\beta_n\rangle$ in $Gal(\mathbb{Q}(\zeta_n,w_n)/\mathbb{Q}(\zeta_1))$. Also, since each $N_{n,i}$ is Abelian (and therefore its own opposite) then $\mathbb{Q}(w_n)/\mathbb{Q}(\zeta_1)$ is Hopf-Galois with respect to the action of $H_{n,i}=(E_{n,i}[N_{n,i}])^{\langle\beta_n\rangle}$ where each is a $E_{n,i}$-form of the group ring $\mathbb{Q}(\zeta_1)[N_{n,i}]$. If we define $\overline{\Delta}_n=\langle\beta_n\rangle$ then we may form the inverse limit $\overline{\Delta}_{\infty}$ of the system $\{\overline{\Delta}_n,\phi_{n,n-1}\}$ in the same fashion as we used to define $\DI$. Similarly, we may define $N_{\infty,i}=\underset{\longleftarrow}{lim}N_{n,i}$ and $E_{\infty,i}=\underset{\longrightarrow}{lim}E_{n,i}$, and $H_{\infty,i}=\underset{\longleftarrow}{\lim}H_{n,i}$. In a manner identical to that developed earlier, we have therefore that $\mathbb{Q}(w_{\infty},\zeta_1)/\mathbb{Q}(\zeta_1)$ is a Hopf-Galois extension with respect to the action of $H_{\infty,i}$, where $H_{\infty,i}\cong(E_{\infty,i}[N_{\infty,i}])^{\overline{\Delta}_{\infty}}$ and $E_{\infty,i}\otimes H_{\infty,i}\cong E_{\infty,i}[N_{\infty,i}]$. This shows that the non-uniqueness of the Hopf-Galois structures which may act on a given extension holds for infinite extensions such as these.
\bibliography{profinite}
\bibliographystyle{plain}
\end{document}